\documentclass[a4paper,12pt]{article}     
\usepackage{amsmath,amscd,amssymb}        
\usepackage{latexsym}                     
\usepackage[english, german]{babel}                

\usepackage{theorem}                 

\newtheorem{theorem}{Theorem}

\newtheorem{proposition}{Proposition}

\theorembodyfont{\rmfamily}
\newtheorem{remark}{Remark}
\newtheorem{example}{Example}

\newenvironment{definition}
{\smallskip\noindent{\bf Definition\/}:}{\smallskip\par}

\newenvironment{proof}{\begin{ProofwCaption}{Proof}}{\end{ProofwCaption}}
\newenvironment{proof*}[1]{\begin{ProofwCaption}{{#1}}}{\end{ProofwCaption}}
\newenvironment{ProofwCaption}[1]%
 {\addvspace\theorempreskipamount \noindent{\it #1.}\rm}%
 {\qed \par \addvspace\theorempostskipamount}
\newcommand{\qedsymbol}{\mbox{$\Box$}}
\newcommand{\qed}{\hfill\qedsymbol}

\newcommand{\CC}{{\mathbb C}}

\newcommand{\RR}{{\mathbb R}}
\newcommand{\ZZ}{{\mathbb Z}}

\newcommand{\eps}{\varepsilon}

\newcommand{\ind}{{\rm ind}\,}
\newcommand{\indGSV}{{\rm ind}_{\rm GSV}}
\newcommand{\indrad}{{\rm ind}_{\rm rad}}
\newcommand{\Conjsub}{{\rm Conjsub}\,}
\newcommand{\Sub}{{\rm Sub}\,}
\newcommand{\Sing}{{\rm Sing}\,}

\title{Equivariant indices of vector fields and 1-forms}
\author{W.~Ebeling and S.~M.~Gusein-Zade
\thanks{Partially supported by the DFG (Mercator fellowship, Eb 102/8--1),
the Russian government grant 11.G34.31.0005, RFBR--13-01-00755,
and NSh--4850.2012.1. 
AMS 2010 Math. Subject Classification: 14B05, 58K45, 58A10, 58E40, 19A22.
}
}
\date{}

\begin{document}
\selectlanguage{english}

\maketitle

\begin{abstract}
Equivariant versions of the radial index
and of the GSV-index of a vector field or a 1-form on a singular variety with an action of a finite group are defined. They have values in the Burnside ring of the group. Poincar\'e--Hopf type theorems for them are proven and some of their properties are described.
\end{abstract}

\section*{Introduction}
The classical notion of the index of an isolated singular point (zero) of a vector field
has several generalizations to vector fields and 1-forms on singular varieties.
In particular, there were defined the GSV-index of a vector field on an isolated hypersurface
or complete intersection singularity \cite{GSV, SS}, the Euler obstruction of a 1-form \cite{MacPh},
the homological index of a vector field \cite{G-MM} or of a 1-form \cite{EGS}, the radial index of a vector
field or of a 1-form, \dots

The radial index of an isolated singular point of a (stratified) vector field or of a 1-form on a singular variety
was defined in \cite{GD, Sur}, following earlier ideas of \cite{MHSch} and \cite{MMJ}.
Initially it was defined for  vector fields and 1-forms on analytic varieties. However without any difference
one can consider the radial index on closed (real) subanalytic varieties (as noticed in \cite{D}). The sum of the
radial indices of a vector field or of  a 1-form with isolated singular points on a compact variety is equal
to the Euler characteristic of the variety. Therefore the (radial) index can be considered as a localization
of the Euler characteristic (on a singular variety).

There are several equivariant versions of the Euler characteristic for spaces with actions of
a finite group $G$. First, one can consider the Euler characteristic of the quotient space.
One can define an equivariant Euler characteristic as an element of the representation ring
of the group $G$: see, e.g., \cite{TtD, Wall1}. There exists the so-called orbifold Euler
characteristic (with values in the ring $\ZZ$ of integers) offered by physicists \cite{Vafa, HH}
and its higher order generalizations \cite{AS, BF}. However a more general concept is the
equivariant Euler characteristic defined as an element of the Burnside ring $B(G)$ of the group $G$
\cite{Verdier, TtD}. The other versions of the Euler characteristic mentioned above are
specializations of this one.

This makes reasonable to try to define indices of singular points of vector fields and of 1-forms
on $G$-varieties as elements of the Burnside ring $B(G)$
so that they will be localizations of the equivariant Euler characteristic.
Their specializations will give other
equivariant versions of the indices, say, as an element of the ring of representations or an
orbifold version (with values in $\ZZ$). For vector fields on (smooth) $G$-manifolds this was
done in \cite{Luck}.

In this paper we define equivariant versions with values in the Burnside ring of the radial index
and of the GSV-index, prove Poincar\'e--Hopf type theorems for them and describe some of their
properties.

In Section~\ref{sect1} we recall the concepts of the radial index and of the GSV-index in the usual
(non-equivariant) setting. In Section~\ref{sect2} we give definitions of the equivariant Euler
characteristic, of the orbifold Euler characteristic and of its higher order generalizations.
In Section~\ref{sect3} we discuss an equivariant version of the radial index. Section~\ref{sect4}
is devoted to an equivariant version of the GSV-index.

\section{Indices of vector fields and 1 forms}\label{sect1}
Let $(V,0)$ be a germ of a closed (real) subanalytic variety. We assume $(V,0)$ to be embedded into $(\RR^N, 0)$.
Let $V=\bigcup_{i=1}^q V_i$ be a subanalytic Whitney stratification of $V$.
A (continuous) {\em stratified vector field} on $V=\bigcup_{i=1}^q V_i$
is a vector field such that, at each point $p$ of $V$, it is tangent
to the stratum containing $p$ \cite{BSS}.

Let $B_\eps$ be a closed ball (in the ambient space $\RR^N$) around the origin
(small enough so that the sphere $S_\eps=\partial B_\eps$ intersects all
the strata $V_i$ transversally). Let $L=V\cap S_\eps$ be the link of $V$.
There exists a diffeomorphism $h$ of $V\cap B_\eps$ with the cone
$CL=L\times[0,1]/L\times\{0\}$ over $L$ being the identity on $L=L\times\{1\}$.
(A diffeomorphism of stratified varieties is a homeomorphism which respects
the stratification and is a diffeomorphism on each stratum.) A vector field $Y$
on $(V,0)$ is called {\em radial} if, for a certain diffeomorphism $h$,
it is identified with the vector field $\frac{\partial{\ }}{\partial t}$ on $CL$.

Let $p\in V$, let $V_{(p)}=V_i$ be the stratum containing $p$, $\dim V_i=k$,
and let $Y$ be a stratified vector field on $V$ in a neighbourhood
of the point $p$. Let $N_i$ be a normal slice of $V$ to the
stratum $V_i$ at the point $p$. There exists a diffeomorphism $h$
from a neighbourhood of $p$ in $V$ to the product $U_i(p)\times N_i$, where
$U_i(p)$ is a neighbourhood of $p$ in $V_i$, which is the identity on $U_i(p)$.
The vector field $Y$ is called a {\em radial extension} of the vector field
$Y_{\vert V_i}$ if there exists a diffeomorphism $h$ which identifies $Y$
with the vector field on $U_i(p)\times N_i$ of the form
$(Y_{\vert V_i},0)+(0, Y^{\rm rad}_{N_i})$, where $Y^{\rm rad}_{N_i}$
is a radial vector field on $N_i$ \cite{BSS}.

Let $X$ be a stratified vector field on $(V,0)$ with an isolated singular point
(zero) at the origin. One can show that there exists a (continuous) $G$-invariant
stratified vector field $\widetilde{X}$ on $V$ such that:
\begin{itemize}
\item[(1)] the vector field $\widetilde{X}$ coincides with $X$ on a neighbourhood
of the intersection of $V$ with the sphere $S_\eps=\partial B_\eps$
of a small radius $\eps$ around the origin;
\item[(2)] the vector field has a finite number of singular points (zeros);
\item[(3)] in a neighbourhood of each singular point $p\in V\cap B_\eps$, $p\in V_i$,
the vector field $\widetilde{X}$ is a radial extension of
its restriction to the stratum $V_i$.
\end{itemize}

The {\em radial index} $\indrad(X;V,0)$ of the vector field $X$
on $V$ at the origin is 
$$
\sum_{{p}\in {\rm Sing}\widetilde{X}} 
{\rm ind\,}(\widetilde{X}_{\vert V_{(p)}};V_{(p)},p)\,,
$$
where ${\rm ind\,}(\widetilde{X}_{\vert V_{(p)}},V_{(p)},p)$ is the usual index of the 
restriction of the vector field $\widetilde{X}$ to the smooth manifold $V_{(p)}$.

Now let $\omega$ be a (continuous) 1-form on $(V,0)$ (i.e.\ the restriction to $V$ of a 1-form
defined in a neighbourhood of $0$ in $\RR^N$). Let $V=\bigcup_{i=1}^q V_i$ be a subanalytic Whitney stratification of $V$. A point $p\in V$ is a {\em singular point} of $\omega$
if the restriction of $\omega$ to the stratum $V_{(p)}$ containing $p$ vanishes at the point $p$.
A 1-form $\eta$ on $(V,0)$ is called {\em radial} if its restriction to any (parametrized) analytic
curve $\varphi:(\RR_{\ge0},0)\to (V,0)$ is positive on the vector $\varphi_*\frac{\partial\ }{\partial t}$ 
in a punctured neighbourhood of zero. As an example, one can take $d \Vert \cdot \Vert^2$, where 
$\Vert \cdot \Vert$ is the Euclidean norm in $\RR^N$.

For a point $p$, $p\in V_i=V_{(p)}$, $\dim V_{(p)}=k$, there exists a (local) analytic
diffeomorphism $h:(\RR^N,\RR^k,0)\to (\RR^N, V_{(p)},p)$. A 1-form $\eta$ is called a
{\em radial extension} of the 1-form $\eta_{\vert V_{(p)}}$ if there exists a diffeomorphism
$h$ which identifies $\eta$ with the restriction to $V$ of the 1-form 
$\pi_1^* \eta_{\vert V_{(p)}}+\pi_2^* \eta^{rad}_{\RR^{N-k},0}$, where $\pi_1$ and $\pi_2$ 
are the projections of $\RR^N$ to $\RR^k$ and $\RR^{N-k}$ respectively,
$\eta^{rad}_{\RR^{N-k},0}$ is a radial 1-form on $(\RR^{N-k},0)$.


For a 1-form $\omega$ on $(V,0)$ with an isolated singular point at the origin there exists
a 1-form $\widetilde{\omega}$ on $V$ which possesses the obvious analogues of the properties (1)--(3)
of the vector field $\widetilde{X}$ above. The {\em radial index} of $\omega$ is 
$$
\indrad(\omega;V,0)=
\sum_{{p}\in {\Sing}\widetilde{\omega}} 
\ind(\widetilde{\omega}_{\vert V_{(p)}};V_{(p)},p)\,.
$$

There is a one-to-one correspondence between complex valued and real valued 1-forms on a complex analytic variety.
The real 1-form $\alpha$ corresponding to a complex 1-form $\omega$ is just its real part: $\alpha={\rm Re}\omega$.
The 1-form $\omega$ can be restored from $\alpha$ by the equation $\omega(u)=\alpha(u)-\alpha(iu)i$
($u$ is a tangent vector to the ambient space). Therefore, 
for a complex 1-form on a complex analytic variety, its radial index is defined as the radial index of its real part. For a complex 1-form on a $d$-dimensional complex analytic manifold, this index differs from the usual one by the sign $(-1)^d$.

The radial index of a vector field or of a 1-form is defined as the sum of the indices
of an appropriate deformation of it on the strata (smooth manifolds). The sum is the same for
a stratification and for a partition of it. This implies that the radial index (and also
its equivariant version below) does not depend on the stratification and thus is well-defined.
(For a vector field one can consider the intersection of two stratifications.
For a 1-form one can consider the minimal Whitney stratification of the variety.)

Let $(V,0)$ be a germ of a subanalytic variety, let $V=\bigcup_{i=1}^q V_i$ be a (subanalytic) Whitney stratification
of $V$, and let $X$ or $\omega$ be a stratified vector field or a 1-form on $(V,0)$ respectively with an isolated
singular point at the origin. One can show that the radial index $\indrad(X; V,0)$ or $\indrad(\omega,V,0)$ can be distributed between the strata $V_i$ in the following sense. The index is the number of singular points of an
appropriate deformation $\widetilde{X}$ or $\widetilde{\omega}$ of $X$ or $\omega$ respectively counted with multiplicities. 

\begin{proposition}\label{prop-distr}
The number of singular points of the vector field $\widetilde{X}$ or of a 1-form $\widetilde{\omega}$
on a fixed stratum $V_i$ does not depend on the choice of the vector field $\widetilde{X}$ or of the 1-form $\widetilde{\omega}$ respectively (and therefore only depends on $X$ or $\omega$ respectively).
\end{proposition}

\begin{proof}
We shall write the proof in terms of vector fields. The proof for 1-forms is the same. We shall follow this scheme in the sequel as well.

Let $\widetilde{X}_1$ and $\widetilde{X}_2$ be two deformations of $X$ with the necessary properties.
Let $W_i$ be the complement in $V_i$ of a small tubular neighbourhood of its boundary $\partial V_i$.
The intersection $W_i\cap B_\eps$ (after an appropriate smoothing at $\partial W_i \cap S_\eps$) is
a manifold with boundary. The vector fields $\widetilde{X}_1$ and $\widetilde{X}_2$ coincide on the part
$W_i\cap S_\eps$ of the boundary of $W_i$. On the remaining part these vector fields are homotopic,
in the class of non-vanishing vector fields, to vector fields pointing inside $W_i$ and therefore are homotopic
to each other (in the same class). This implies that the numbers of the singular points of the vector fields
$\widetilde{X}_1$ and $\widetilde{X}_2$ on $W_i$ counted with multiplicities coincide.

Another argument for the statement can be the following one.
The closure $\overline{V}_i$ of the stratum $V_i$ and its boundary $\overline{V}_i\setminus V_i$
are closed subanalytic varieties. The numbers of the singular points of the vector field $\widetilde{X}$
on $\overline{V}_i$ and on $\overline{V}_i\setminus V_i$ (counted with multiplicities) are equal
to the indices $\indrad(X;\overline{V}_i,0)$ and $\indrad(X;{\overline{V}_i\setminus V_i,0})$ respectively
(and therefore are well-defined). The number of the singular points of $\widetilde{X}$
on ${V}_i$ is the difference 
$$
\indrad(X;\overline{V}_i,0)-\indrad(X;{\overline{V}_i\setminus V_i,0}).
$$
\end{proof}

This means that the radial index can be defined for each stratum $V_i$ so that
$$
\indrad(X;V,0)=\sum_{i=1}^q \indrad(X; V_i,0)\,.
$$
Moreover the radial index can be defined as a linear map from the set of constructible,
integer valued functions on $V$ constant on each stratum $V_i$ to $\ZZ$. (The radial index
$\indrad(X;V,0)$ is the image of the constant function $1$.)

There is also the notion of the GSV-index of a vector field or a 1-form on a complex analytic complete intersection singularity. 
We shall discuss this notion (and therefore its $G$-equivariant version below)  in the generality considered in \cite{BSS}.
Let $f_1, \ldots, f_k$ be holomorphic
function germs on $(\CC^n,0)$ defining a complete intersection singularity 
$(V,0)=\{f_1=\ldots=f_k=0\}$. We assume that a neighbourhood of the origin in  $\CC^n$ permits a Whitney stratification adapted to V and satisfying the Thom $a_f$ condition. This holds, in particular, if $(V,0)$ is an
isolated complete intersection singularity (ICIS) or if it is a hypersurface, i.e.\ $k=1$. Let $F$ be the map
$F=(f_1, \ldots, f_k): (\CC^n,0)\to(\CC^k,0)$ and let $\Delta\subset (\CC^k,0)$ be the discriminant of $F$.
For $0<\delta\ll\eps$ small enough the restriction of the map $F$ to $B_{\eps}^{2n}$ is a locally trivial fibration
over $B_{\delta}^{2k}\setminus\Delta$ ($B_{\eps}^{2n}$ and $B_{\delta}^{2k}$ are the closed balls
of radii $\eps$ and $\delta$ around the origin in $\CC^n$ and in $\CC^k$ respectively.).
For $\underline{\delta}=(\delta_1,\ldots, \delta_k)\in B_{\delta}^{2k}\setminus\Delta$, the fibre
$M_{\underline{\delta}}={F}^{-1}(\underline{\delta})\cap B_{\eps}^{2n}$ of this fibration is a non-singular complex manifold with the boundary $\partial M_{\underline{\delta}}=F^{-1}(\underline{\delta})\cap S_{\eps}^{2n-1}$~---
the {\em Milnor fibre} of $F$ or of $(V,0)$ (the set $S_{\eps}^{2n-1}=\partial B_{\eps}^{2n}$
is the sphere of radius $\eps$ around the origin in $\CC^n$).

Let $X$ be a stratified vector field on $(V,0)$ with an isolated singular point at the origin. In particular,
the restriction of $X$ to a neighbourhood of $V\cap S_{\eps}^{2n-1}$ has no zeros.
In \cite{BSS} it is explained how one can construct a vector field $\widehat{X}$ on a neighbourhood of the boundary
$\partial M_{\underline{\delta}}$ of the Milnor fibre approximating $X$ on a neighbourhood of
$V\cap S_{\eps}^{2n-1}$. Two approximations of this sort
are close to each other and therefore homotopic to each other in the class of non-vanishing vector
fields. Let $\widetilde{X}$ be an extension of  $\widehat{X}$ to the Milnor fibre $M_{\underline{\delta}}$
with isolated singular points.

The GSV-{\em index} of the vector field $X$ is
\begin{equation*}\label{GSV}
\indGSV(X;V,0)=
\sum_{p \in\Sing \widetilde{X}}
\ind(\widetilde{X}; M_{\underline{\delta}},p)\,.
\end{equation*}

\section{Burnside rings and Euler characteristics}\label{sect2}

Let $G$ be a finite group. Let $\Conjsub G$ be the set of conjugacy classes of subgroups of $G$.
A $G$-set is a set with an action of the group $G$.
A $G$-set is {\em irreducible} if the action of $G$ on it is transitive.
Isomorphism classes of irreducible $G$-sets are in one-to-one correspondence with
the elements of $\Conjsub G$: to the conjugacy class $[H]$ containing
a subgroup $H\subset G$ one associates the isomorphism class $[G/H]$
of the $G$-set $G/H$. The {\em Burnside ring} $B(G)$ of $G$ is the
Grothendieck ring of finite $G$-sets, i.e.\
the abelian group generated by the isomorphism classes of finite $G$-sets
modulo the relation $[A\amalg B]=[A]+[B]$ for finite $G$-sets $A$ and $B$. The multiplication
in $B(G)$ is defined by the 
cartesian product. As an abelian group, $B(G)$
is freely generated by the isomorphism classes of irreducible $G$-sets,
i.e.\ each element of $B(G)$ can be written in a unique way as 
$\sum\limits_{[H]\in \Conjsub G}a_{[H]}[G/H]$ with $a_{[H]}\in\ZZ$.
The element $1$ in
the ring $B(G)$ is represented by the $G$-set $[G/G]$ consisting of one point (with the trivial $G$-action).

There is a natural homomorphism from the Burnside ring $B(G)$ to the ring $R(G)$
of representations of the group $G$ which sends a $G$-set $X$ to the (vector)
space of functions on $X$. If $G$ is cyclic, then this homomorphism is injective.
In general, it is neither injective nor surjective.

For a subgroup $H\subset G$ there are natural maps
$\mbox{R}_{H}^{G}: B(G)\to B(H)$ and $\mbox{I}_{H}^{G}: B(H)\to B(G)$.
The {\em restriction map} $\mbox{R}_{H}^{G}$ sends a $G$-set X to the same set considered with the $H$-action.
The {\em induction map} $\mbox{I}_{H}^{G}$ sends an $H$-set $X$ to the product $G\times X$ factorized
by the natural equivalence: $(g_1,
 x_1)\equiv (g_2, x_2)$ if there exists $g\in H$ such that
$g_2=g_1g$, $x_2=g^{-1}x_1$ with the natural (left) $G$-action. Both maps are group homomorphisms,
however the induction map $\mbox{I}_{H}^{G}$ is not a ring homomorphism. 

For an action of a group $G$ on a set $X$ and for a point $x\in X$, let $G_x=\{g\in G: gx=x\}$
be the isotropy group of the point $x$. For a subgroup $H\subset G$ let
$X^H=\{x\in X: Hx=x\}$ be the fixed point set of the subgroup $H$
and let $X^{(H)}=\{x\in X: G_x=H\}$ be the set of points with the isotropy group $H$.
For a conjugacy class $[H]\in \Conjsub G$, let
$X^{[H]}=\bigcup\limits_{K\in [H]}X^{K}$,
$X^{([H])}=\bigcup\limits_{K\in [H]}X^{(K)}$.

For a ``sufficiently good'' topological space $V$, say, a subanalytic variety,
its Euler characteristic is
$$
\chi(V)=\sum_i (-1)^i \dim H^i_c(V;\RR),
$$
where $H^i_c(V;\RR)$ are the cohomology groups with compact support. This Euler characteristic
is additive (in contrast to the one defined via the usual cohomology groups).
For quasi-projective complex analytic varieties both these Euler characteristics coincide. 

Let $G$ be a finite group acting on $V$.

\begin{definition}
 The {\em equivariant Euler characteristic} of $(V,G)$ is the element of the Burnside ring of $G$
 defined by
 $$
 \chi^G(V)=\sum_{{[H]}\in\Conjsub G}\chi(V^{([H])}/G)[G/H]
 $$
 (see, e.g., \cite{Verdier, TtD}). The {\em reduced equivariant Euler characteristic} of $(V,G)$
 is
 $$
 {\overline{\chi}}^G(V)=\chi^G(V)-[G/G]\,.
 $$
\end{definition}

\begin{remark} For a subgroup $H$ from a conjugacy class $[H]$, the $G$-action
 on $V$ induces an $N_G(H)$-action on $V^H$ ($N_G(H)$ is the normalizer $\{g\in G: g^{-1}Hg=H\}$ of $H$)
 and $V^{([H])}/G=V^{({H})}/N_G(H)$. Therefore
$$
\chi^G(V)=\sum\limits_{{[H]}\in\Conjsub G}\chi(V^{(H)}/N_G(H))[G/H] .
$$
\end{remark} 

\begin{remark}
 One can see that the natural homomorphism from $B(G)$ to the ring $R(G)$ of representations of $G$
described above sends the equivariant Euler characteristic $\chi^G(V)$ to the equivariant Euler characteristic $\chi_G(V)\in R(G)$ considered, e.g., in \cite{Wall1}.
\end{remark}

There is a natural homomorphism $\vert\cdot\vert$ from the Burnside ring $B(G)$ to $\ZZ$ sending a (virtual) $G$-set
$A$ to the number of elements $\vert A\vert$ of $A$. This homomorphism sends the equivariant Euler characteristic
$\chi^G(V)$ to the usual Euler characteristic $\chi(V)$.

The {\em orbifold Euler characteristic} of the pair $(V,G)$ is defined by
$$ 
\chi^{orb}(V,G)=\frac{1}{\vert G\vert}\sum_{(g,h):gh=hg} \chi(X^{\langle g,h\rangle}) 
$$ 
where $\langle g,h\rangle$ is the subgroup of $G$ generated by $g$ and $h$ (see \cite{Vafa, HH, AS}). 

There are {\em higher order Euler characteristics} of the pair $(V,G)$ defined by
$$ 
\chi^{(k)}(V,G)=\frac{1}{\vert G\vert}
\sum_{{{\bf g}\in G^{k+1}:}\atop{g_ig_j=g_jg_i}} \chi(X^{\langle g_0, g_1, \ldots, g_k\rangle}), \quad k\ge 0
$$ 
where ${\bf g}=(g_0,g_1, \ldots, g_k)$, $\langle g_0, g_1, \ldots, g_k\rangle$ is the subgroup 
generated by $g_0,g_1, \ldots, g_k$ (see \cite{AS, BF}). The zeroth order Euler characteristic $\chi^{(0)}(V,G)$ of the pair $(V,G)$ is nothing else but the Euler characteristic
of the quotient $V/G$, the first order Euler characteristic $\chi^{(1)}(V,G)$ of the pair $(V,G)$ is just the orbifold one.
All these Euler characteristics are additive functions on the Grothendieck ring of subanalytic $G$-varieties.

There are homomorphisms $r_G^{(k)}$ from $B(G)$ as an abelian group to $\ZZ$ which send the equivariant Euler
characteristic $\chi^G(V)$ to the higher order Euler characteristics $\chi^{(k)}(V,G)$, $k\ge 0$. The homomorphism
$r_G^{(k)}$ is defined by $r_G^{(k)}([G/H])=\chi^{(k)}([G/H],G)$. The homomorphism $r_G^{(0)}$ sends the class $[A]$
of a finite $G$-set to the number of $G$-orbits in $A$. For an abelian group $G$ one has
$\chi^{(k)}([G/H],G)=\vert H\vert^k$. Therefore one can say that the higher order Euler characteristics are reductions of the equivariant Euler characteristic which is a more general one. The homomorphism $r_G^{(k)}$ is not, in general, a ring homomorphism. It is a ring homomorphism for abelian $G$ and $k=1$ (i.e.\ for
the orbifold Euler characteristic).

Below we construct $G$-equivariant versions of some integer valued invariants as elements of the
Burnside ring $B(G)$. They possess the property that the ``number of elements'' homomorphism $\vert\cdot\vert$
sends them to the original invariants (i.e.\ for trivial $G$). Applying the homomorphism $r_G^{(k)}$ one gets an orbifold version of each invariant (for $k=1$)
and also higher order integral invariants. (This was used in \cite{GLM} to define orbifold versions of the Lefschetz number.)

\section{Equivariant radial index}\label{sect3}
We introduce the notion of an equivariant radial index for vector fields. The necessary changes for 1-forms are clear.

Let $(V,0)$ be a germ of a closed real subanalytic variety with an action
of a finite group $G$. (We assume $(V,0)$ to be embedded into $(\RR^N, 0)$
and the $G$-action to be induced by an analytic action on a neighbourhood of
$(V,0)$ in $(\RR^N, 0)$. This essentially means that the action on $(\RR^N, 0)$
can be assumed to be linear.) 
Let $V=\bigcup_{i=1}^q V_i$ be a subanalytic Whitney
$G$-stratification of $V$. This means that each stratum $V_i$ is $G$-invariant,
the isotropy subgroups $G_p=\{g\in G: gp=p\}$ of all points $p$ of $V_i$ are conjugate to each other,
and the quotient of the stratum by the group $G$
is connected. (Such a stratification can be obtained from an arbitrary subanalytic Whitney
stratification of $(V,0)$ by taking intersections/unions of the strata and
of their transforms by the elements of $G$.)

Let $X$ be a $G$-invariant stratified vector field on $(V,0)$ with an isolated singular point
(zero) at the origin. One can show that there exists a (continuous) $G$-invariant
stratified vector field $\widetilde{X}$ on $V$ satisfying conditions (1)--(3) from Section~\ref{sect1}.

Let $A$ be the set (a $G$-set) of the singular points of the vector field
$\widetilde{X}$ on $V\cap B_\eps$ considered with the multiplicities equal to
the usual indices $\ind(\widetilde{X}_{\vert V_{(p)}}; V_{(p)},p)$ of the restrictions
of the vector field $\widetilde{X}$ to the corresponding strata (smooth manifolds).

\begin{definition}
 The {\em equivariant radial index} $\indrad^G(X; V,0)$ of the vector field $X$
 on $V$ at the origin is the class $[A]\in B(G)$ of the set $A$ of singular
 points of $\widetilde{X}$ with multiplicities.
\end{definition}

\begin{remark}
 One can rewrite the definition as
\begin{equation*}\label{Defrad}
\indrad^G(X; V,0)=\sum_{\overline{p}\in ({\Sing}\widetilde{X})/G} 
\ind(\widetilde{X}_{\vert V_{(p)}}; V_{(p)},p) [Gp]\,, 
\end{equation*}
 where $p$ is a point of the preimage of $\overline{p}$ under the canonical projection.
\end{remark}

For a subgroup $H\subset G$, the vector field $X$ is $H$-invariant and one has
$\indrad^H(X; V,0)=R^G_H(\indrad^G(X; V,0))$.

The homomorphisms $r_G^{(k)}$ give "higher order Euler characteristic versions" of the equivariant index:
$\indrad^{G,(k)}(X; V,0)=r_G^{(k)}(\indrad^G(X; V,0))$. For $k=1$ one gets an orbifold index of $X$.

For a stratum $V_i$ of the stratification,
let $[G_i]\in \Conjsub G$ be the conjugacy class of the isotropy subgroups of points
of $V_i$, and let $G_i$ be a representative of it. Singular points of the vector field $\widetilde{X}$
on the stratum $V_i$ are distributed between the orbits of the $G$-action, all of which (as $G$-sets)
are isomorphic to $[G/G_i]$. Therefore
\begin{equation}\label{equi1}
\indrad^G(X; V,0)= \sum_{i=1}^q \frac{\vert G_i\vert}{\vert G\vert}\ind(X;V_i,0) [G/G_i]
\end{equation}
(see Proposition~\ref{prop-distr}).
 
Let us write the equivariant index $\indrad^G(X; V,0)$ in terms of the usual (non-equivariant)
indices of the vector field $X$ on the fixed point sets $V^H$ of subgroups of $G$ or
on the ``fixed point sets'' $V^{[H]}=\bigcup\limits_{K\in [H]}V^K$
of conjugacy classes of subgroups of $G$.

Assume that
\begin{equation}\label{equi2}
 \indrad^G(X; V,0) = \sum_{[H]\in\Conjsub G} a_{[H]} [G/H]\,.
\end{equation}
One has
\begin{equation*}
 \indrad(X;V,0) = \sum_{[H]\in\Conjsub G} a_{[H]} \frac{\vert G\vert}{\vert H\vert} \,.
\end{equation*}
Equation~\ref{equi2} implies that
\begin{equation*}
 \indrad^G(X;V^{[H]},0) = \sum_{[K]\ge [H]} a_{[K]} [G/K]
\end{equation*}
and therefore
\begin{equation}\label{equi3}
\indrad(X;V^{[H]},0)  =
 \sum_{[K]\ge[H]} a_{[K]} \frac{\vert G\vert}{\vert K\vert}\,.
\end{equation}
Equation~\ref{equi3} can be rewritten as 
\begin{equation}\label{equi4}
\indrad(X;V^{[H]},0) =
 \sum_{[K]\in\Conjsub G} \zeta([H],[K])a_{[K]}
 \frac{\vert G\vert}{\vert K\vert}\,,
\end{equation}
 where
\begin{equation*}
 \zeta([H],[K])=\left\{ \begin{array}{cl} 0 & \mbox{for\ }[K]\not\ge[H],\\
 1 & \mbox{for\ }[K]\ge [H] \end{array} \right.
\end{equation*}
is the zeta function of the partially ordered set $\Conjsub G$. (One has $[K]\ge [H]$
if there exist representatives $K$ and $H$ of $[K]$ and $[H]$ respectively such that
$K\supset H$.)
The system~(\ref{equi4}) is a triangular system of linear equations with respect to $a_{[H]}$.
Its solution gives the following statement.

\begin{proposition}
One has 
\begin{equation*}
 \indrad^G(X; V,0) = \sum_{[H]\in\Conjsub G}\frac{\vert H\vert}{\vert G\vert}
\left( \sum_{[K]\in\Conjsub G}\mu([H], [K]) \indrad(X;V^{[K]},0) \right)[G/H]\,,
\end{equation*}
where $\mu([H],[K])$ is the Moebius function of the poset $\Conjsub G$ (see, e.g, \cite{Hall}).
\end{proposition}

In a similar way one has
\begin{equation}\label{equi5}
 \indrad(X;V^{H},0) =
 \sum_{K\ge H} a_{[K]} \frac{\vert N_G(K)\vert}{\vert K\vert}\,,
\end{equation}
where $N_G(K)$ is the normalizer $\{g\in G:g^{-1}Kg=K\}$ of the subgroup $K$.
This can be rewritten as 
\begin{equation}\label{equi6}
 \indrad(X;V^{H},0) =
 \sum_{K\in\Sub G} \zeta'(H,K)a_{[K]}
 \frac{\vert N_G(K)\vert}{\vert K\vert}\,,
\end{equation}
 where
$\zeta'(H,K)$
is the zeta function of the partially ordered set $\Sub G$ of subgroups of $G$.
Thus one gets the following statement.

\begin{proposition} \label{propsubgroups}
One has 
\begin{equation*}
 \indrad^G(X;V,0) = \sum_{[H]\in\Conjsub G}\frac{\vert H\vert}{\vert N_G(H)\vert}
 \left(\sum_{K\in\Sub G}\mu'(H,K) \indrad(X;V^{K},0)\right)[G/H]\,,
\end{equation*}
where $\mu'(H,K)$ is the Moebius function of the poset $\Sub G$.
\end{proposition}

Let $V$ be a locally compact subanalytic $G$-variety 
and let $X$ be a $G$-invariant (stratified with respect to a stratification) vector field on $V$ with isolated
singular points. Let $p$ be a singular point of $X$. The germ $(V,p)$ caries an action of the isotropy subgroup $G_p$
of $p$. Therefore the index $\indrad^{G_p}(X;V,p)\in B(G_p)$ is defined.

\begin{definition}
The {\em $G$-equivariant index} $\indrad^G(X;V,Gp)$ of the vector field $X$ at the orbit $Gp$ is
 $I_{G_p}^G (\indrad^{G_p}(X;V,p)) \in B(G)$.
\end{definition}

One can see that the $G$-equivariant index of a vector field (and of a 1-form as well)
satisfies the law of conservation of number:
for a ($G$-invariant) deformation $\widetilde{X}$ of a vector field $X$ the sum of the $G$-equivariant
indices of $\widetilde{X}$ at the orbits in a neighbourhood of an orbit $Gp$ is equal to
$\indrad^G(X; V,Gp)$.

One has the following equivariant version of the Poincar\'e-Hopf Theorem.

\begin{theorem}
Let $V=\bigcup_{i=1}^q V_i$ be a compact subanalytic variety and let $X$ be a $G$-invariant stratified vector
field on $V$ with isolated singular points. Then the sum of the indices of the vector field $X$ over all singular orbits
is equal to the $G$-equivariant Euler characteristic of $V$:
$$
\sum_{\overline{p}\in (\Sing X)/G} \indrad^G(X;V,\pi^{-1}(\overline{p})) =\chi^G(V)\in B(G)
$$
($\pi:V\to V/G$ is the canonical projection).
\end{theorem}

An analogous statement holds for 1-forms.

\begin{proof}
Let $\widetilde{X}$ be a $G$-invariant deformation of $X$ which differs from $X$ only in neighbourhoods of
singular points and which satisfies
 the properties of the definition of the equivariant index. One has
\begin{eqnarray*}
\sum_{\overline{p}\in (\Sing X)/G} \indrad^G(X,V,\pi^{-1}(\overline{p})) &=&
\sum_{\overline{p}\in \Sing X/G} \indrad^G(\widetilde{X};V,\pi^{-1}(\overline{p})) \\
& = & 
\sum_{i=1}^q \sum_{p\in V_i}\frac{\vert G_i\vert}{\vert G\vert} \ind(\widetilde{X};V_i,p) [G/G_i]\,, \\
\chi^G(V) &= & \sum_{i=1}^q \frac{\vert G_i\vert}{\vert G\vert} \chi(V_i)[G/G_i]\,.
\end{eqnarray*}

As in the proof of Proposition~\ref{prop-distr}, let $W_i$ be the complement in $V_i$ of a small tubular neighbourhood
of its boundary $\partial V_i$. It can be regarded as a manifold with boundary. On the boundary 
$\partial W_i$ the restriction
$\widetilde{X}_{\vert W_i}$ of the vector field $\widetilde{X}$ is homotopic, in the
class of non-vanishing vector fields, to a vector field pointing inwards. Therefore 
$$
\sum_{p\in V_i}\ind(\widetilde{X};V_i,p)= \chi(W_i\setminus \partial W_i)=\chi(V_i)\,.
$$
(We use the fact that, on a compact manifold with boundary, the sum of the indices of the
singular points of a vector field
pointing inwards on the boundary is equal to the Euler characteristic of the interior of the manifold.)
\end{proof}

Let $(V,0)$ be a germ of a complex analytic variety of pure dimension $n$ and let $f:(V,0)\to(\CC,0)$ be a germ of a complex analytic function. The germ $f$ is the lifting $\overline{f}\circ \pi$ of an analytic function  germ
$\overline{f}:(V/G,0)\to(\CC,0)$ ($\pi:V\to V/G$ is the canonical projection).
The Milnor fibres $M_{{\rm Re} f}^+$ and $M_{{\rm Re} f}^-$ of ${\rm Re} f$ are $G$-homotopy equivalent to the Milnor fibre
$f^{-1}(\delta)\cap B_\eps$ of $f$ and one has $\indrad^G(df;V,0)=-\overline{\chi}^G(M_f)$.
The element $\mu^G_f=(-1)^{n-1}\overline{\chi}^G(M_f)\in B(G)$
can be regarded as an equivariant version of the Milnor number of $f$. The homomorphism $r_G^{(0)}$ from $B(G)$ to $\ZZ$ which sends $[G/H]$ to 1 maps the element $\mu^G_f$ to the Milnor number of the germ $\overline{f}$ on
$(V/G, 0)$ defined (for smooth $V/G$) in \cite{Wall1}.

Let $(V,0)$ be a germ of a closed subanalytic $G$-variety, let $(V/G,0)$ be the quotient by the $G$-action,
let $\pi:V\to V/G$ be the canonical projection and let $V=\bigcup\limits_{i=1}^q V_i$ and 
$V/G=\bigcup\limits_{i=1}^q V_i/G$ be Whitney stratifications of $V$ and $V/G$ respectively compatible with the
partitions corresponding to different orbit types. 
For a stratum $V_i$, $i=1,\ldots, q$, let $[G_i]$ be the conjugacy class of the isotropy subgroups of points of $V_i$
and let $G_i$ be a representative of it. 
Let $X$ be a $G$-invariant stratified vector field on $V$
with an isolated singular point at $0$. The image of $X$ under the map $\pi_*$ defines a stratified vector field
$\overline{X}$ on $V/G$ with an isolated singular point. 
One can easily see that a product/cone-like structure in a neighbourhood of a point $\overline{p}\in V_i/G$
defines a $G_p$-invariant product/cone-like structure in a neighbourhood of a point $p\in \pi^{-1}(\overline{p})$.
This implies that, if $\widetilde{X}$ is an appropriate (in the sense of the definition of the radial index) deformation of the vector
field $X$, then $\pi_*\widetilde{X}$ is an appropriate deformation of the vector
field $\overline{X}$. Therefore one has the following statement.

\begin{proposition} \label{prop-vect}
$$ 
\indrad^G(X;v,0)=\sum_{[H]\in\Conjsub G} \indrad(\overline{X};V^{([H])}/G, 0)[G/H]\,,
$$
where $\indrad(\overline{X};V^{([H])}/G, 0)$ is the (radial) index of the vector field $\overline{X}$ on the
non-closed variety $V^{([H])}/G$ defined in Section~\ref{sect1}.
\end{proposition}

Let $\omega$ be a 1-form on $V/G$ with an isolated singular point at $0$. The 1-form $\pi^*\omega$ on $V$
is $G$-invariant and has an isolated singular point at $0$. One can easily see that the following analogue
of Proposition~\ref{prop-vect} holds.

\begin{proposition}
$$
\indrad^G(\pi^*\omega; V,0)=\sum_{[H]\in\Conjsub G} \indrad(\omega;V^{([H])}/G, 0)[G/H]\,.
$$
\end{proposition}

\begin{example} Assume that $f:(V,0)\to(\RR,0)$ is a $G$-invariant germ of an analytic function on $(V,0)$
with an isolated critical point at $0$. (This means that the differential $df$ of $f$ has an isolated singular point at $0$.)
For a small enough positive $\eps$ the restriction of $f$ to $V\cap B_{\eps}$ is a locally trivial fibration over a punctured neighbourhood of zero in $\RR$, in general with different fibres over the positive and the negative
part. The fibres ($G$-varieties) $M_f^+$ and $M_f^-$ of this fibration can be called positive and negative Milnor fibres of $f$ respectively. One has the following generalization of \cite[Theorem 2]{EGS}.
\end{example}

\begin{proposition}\label{ind-mu-form}
$$
\indrad^G(df;V,0)=-\overline{\chi}^G(M_f^-)\,.
$$
\end{proposition}

For a ($G$-invariant) complex analytic function $f:(V,0)\to(\CC,0)$ on a germ of a complex analytic variety,
the Milnor fibres $M_f^{\pm}$ of its real part are homeomorphic to the product $M_f\times I$, where
$M_f=f^{-1}(\delta)\cap B_{\eps}$ is the Milnor fibre of $f$ ($0<\vert\delta\vert\ll\eps$ small enough),
$I$ is the segment $[0,1]$: see \cite[Proposition 2.A.3]{GM}. Therefore in this case 
$$
\indrad^G(df;V,0)=-\overline{\chi}^G(M_f)\,.
$$

\begin{example}\label{ex-dim}
Let the complex vector space $\CC^n$ be equipped with a (linear) $G$-action and let
$\omega=\sum\limits_{i=1}^n A_i(z)dz_i$ be a complex analytic $G$-invariant 1-form on $(\CC^n, 0)$
with an isolated singular point at the origin. For a subgroup $K\subset G$, let $n_K$ be the dimension of the
fixed point set $(\CC^n)^K$ of $K$ (a vector subspace). The (radial) index $\indrad(\omega; (\CC^n)^K,0)$
is $(-1)^{n_K}$ times the dimension of the space
$\Omega_{(\CC^n)^K,\, \omega}:=\Omega_{(\CC^n)^K,0}^{n_K}/(\Omega_{(\CC^n)^K,0}^{n_K-1}\wedge \omega_{\vert(\CC^n)^K})$, where
$\Omega_{(\CC^n)^K,0}^{m}$ is the vector space of germs of analytic differential $m$-forms on $((\CC^n)^K,0)$.
Together with Proposition~\ref{propsubgroups} this gives
$$
\indrad^G(\omega;\CC^n,0)=
$$
$$
\sum_{[H]\in\Conjsub G}\frac{\vert H\vert}{\vert N_G(H)\vert}
\left(\sum_{K\in {\Sub G}}(-1)^{n_K}\mu'(H,K)
\dim\Omega_{(\CC^n)^K,\,\omega}\right)
[G/H]\,.
$$
\end{example}

\begin{example}
Let $f(z_1, \ldots, z_n)=\sum_{i=1}^n \prod_{j=1}^n z_i^{E_{ij}}$ be an invertible polynomial (i.e.\ $E_{ij} \in \ZZ$, $E_{ij} \geq 0$, $\det(E_{ij})\ne 0$) with an isolated critical point at the origin. Let 
$$
G_f=\{(\lambda_1, \ldots, \lambda_n)\in\CC^n:
f(\lambda_1 z_1, \ldots, \lambda_n z_n)=f(z_1, \ldots, z_n)\}
$$ 
be the group of diagonal symmetries of $f$.
The Berglund-H\"ubsch dual polynomial to $f$ is
$\widetilde{f}(z_1, \ldots, z_n)=\sum_{i=1}^n \prod_{j=1}^n z_i^{E_{ji}}$.
The group $G_{\widetilde{f}}$ of diagonal symmetries of $\widetilde{f}$ is (canonically) isomorphic
to $G_f^*={\rm Hom}(G,\CC^*)$ \cite{BLMS}. The result of \cite{BLMS} together with Proposition~\ref{ind-mu-form} says
that the indices $\indrad^{G_f}(df;\CC^n,0)$ and $\indrad^{G_{\widetilde{f}}}(d\widetilde{f};\CC^n,0)$
are Saito dual to each other in the sense of \cite{BLMS}.
Let $\overline{f}$ and $\overline{\widetilde{f}}$ be the functions on $\CC^n/G_f$ and
$\CC^n/G_{\widetilde{f}}$ corresponding to $f$ and $\widetilde{f}$ respectively. The Saito dual
elements of the Burnside rings $B(G_f)$ and $B(G_{\widetilde{f}})$ have the same images
under $r^{(0)}_{G_f}$ and $r^{(0)}_{G_{\widetilde{f}}}$ respectively. Therefore the indices
$\indrad^{G_f,(0)}(df;\CC^n,0)=\indrad(d\overline{f}; \CC^n/G_f,0)$ and
$\indrad^{G_{\widetilde{f}},(0)}(d\widetilde{f};\CC^n,0)=
\indrad(d\overline{\widetilde{f}}; \CC^n/G_{\widetilde{f}},0)$ (equal to the minus reduced Euler
characteristics of $M_f/G_f$ and $M_{\widetilde{f}}/G_{\widetilde{f}}$ respectively) coincide.
For a subgroup $H\subset G_f$ its dual  $H^T\subset G_{\widetilde{f}}$ is the kernel of the natural map $i^*: G_f^* \to H^*$ induced by the inclusion
$i : H \to G_f$. The result of \cite{MMJ2} says that the orbifold indices $\indrad^{H,(1)}(df;\CC^n,0)$ and $\indrad^{H^T,(1)}(d\widetilde{f};\CC^n,0)$ coincide.
\end{example}

\section{Equivariant GSV-index}\label{sect4}

Let $(\CC^n,0)$ be endowed by a $G$-action. (Without loss of generality one may assume that the action is linear,
i.e.\ it is defined by a representation $G\to {\rm GL}(n)$.) Let $f_1, \ldots, f_k$ be $G$-invariant holomorphic
function germs on $(\CC^n,0)$ defining a complete intersection singularity 
$(V,0)=\{f_1=\ldots=f_k=0\}$. We assume that a neighbourhood of the origin in  $\CC^n$ permits a Whitney stratification adapted to V and satisfying the Thom $a_f$ condition.  

Let $X$ be a $G$-invariant stratified vector field on $(V,0)$ with an isolated singular point at the origin. In particular,
the restriction of $X$ to a neighbourhood of $V\cap S_{\eps}^{2n-1}$ has no zeros.
The vector fields $\widehat{X}$ and $\widetilde{X}$ described in Section~\ref{sect1} can be made $G$-invariant. (To get a $G$-invariant vector field, one can take an arbitrary one and take the mean over the group.)

\begin{definition}
The {\em $G$-equivariant GSV-index} of the vector field $X$ is
\begin{equation}\label{G-GSV}
\indGSV^G(X;V,0)=
\sum_{\overline{p}\in (\Sing \widetilde{X})/G}
\indrad^G (\widetilde{X}; M_{\underline{\delta}},\pi^{-1}(\overline{p}))\in B(G),
\end{equation}
where $\pi:M_{\underline{\delta}}\to M_{\underline{\delta}}/G$ is the canonical projection.
\end{definition}

It is easy to show that the right hand side of (\ref{G-GSV}) does not depend on the extension $\widetilde{X}$
of $\widehat{X}$ and therefore the equivariant GSV-index is well-defined.

A compact complex analytic $G$-variety $V$, all singular points of which are ICIS, has a natural smoothing
$\widetilde{V}$. (The manifold $V$ is in general $C^{\infty}$, not necessarily complex analytic.
It can be chosen to be complex analytic if $V$ is a projective complete intersection with
isolated singularities.) Let $X$ be a $G$-invariant vector field on $V$ with isolated singular points.
For $p\in \Sing X$, the $G$-equivariant GSV-index of $X$ at the orbit $Gp$ is defined as
$$
\indGSV^G(X; V,Gp):=I_{G_p}^G (\indGSV^{G_p}(X;V,p))\,.
$$
Since the GSV-index counts singular points of the vector field on the 
smoothing of the ICIS, one has the following version of the Poincar\'e--Hopf Theorem.

\begin{proposition}
The sum of the equivariant GSV-indices of the singular orbits of $X$ on $V$ is equal to the 
equivariant Euler characteristic of $\widetilde{V}$:
$$
\sum_{Gp\in(\Sing X)/G} \indGSV^G(X;V,Gp) = \chi^G(\widetilde{V})\in B(G)\,.
$$
\end{proposition}

A relation between the equivariant GSV-index and the radial one can be described in the following way.
The Milnor fibre $M_{\underline{\delta}}$ is a manifold with a $G$-action and therefore its equivariant
Euler characteristic $\chi^G(M_{\underline{\delta}})\in B(G)$ is defined. Let 
$\overline{\chi}^G(M_{\underline{\delta}})=\chi^G(M_{\underline{\delta}})-1$ be the reduced
equivariant Euler caracteristic of $M_{\underline{\delta}}$. (The element
$(-1)^{n-k} \overline{\chi}^G(M_{\underline{\delta}})\in B(G)$ can be regarded as an equivariant version
of the Milnor number of the ICIS $(V,0)$.)

\begin{proposition} \label{propGSV-rad}
One has
$$
\indGSV^G(X;V,0)=\indrad^G(X;V,0)+\overline{\chi}^G(M_{\underline{\delta}})\,.
$$
\end{proposition}

\begin{proof}
 This follows from the following facts: 1) both the GSV-index and the radial one 
 satisfy the law of conservation of number, and 2) for the radial vector field
 $X_{rad}$ on $(V,0)$ one has $\ind^G(X_{rad};V,0)=1$,
 $\indGSV^G(X_{rad};V,0)=\overline{\chi}^G(M_{\underline{\delta}})$.
\end{proof}

Analogous notions and statements exist for $G$-invariant 1-forms on  complete intersection singularities.

\begin{example}
 Let $(V,0)=\{f_1=\ldots=f_k=0\}$ be an ICIS. The equivariant Euler characteristic
 $\chi^G(M_{\underline{\delta}})$ of the Milnor fibre of $(V,0)$ is equal to the sum
 of the equivariant indices of the $G$-orbits of singular points on $M_{\underline{\delta}}$
 of a radial 1-form on $V$. Applying Proposition~\ref{propsubgroups} to each of the indices one gets:
 \begin{equation*}
 \chi^G(M_{\underline{\delta}}) = \sum_{[H]\in\Conjsub G}\frac{\vert H\vert}{\vert N_G(H)\vert}
 \left(\sum_{K\in\Sub G}\mu'(H,K) \chi(M_{\underline{\delta}})\right)[G/H]\,.
\end{equation*}
Since $\sum\limits_{K}\mu'(H,K)=\sum\limits_{K}\zeta(K,G)\mu'(H,K)=\delta(H,G)$, where $\delta(\cdot , \cdot)$ is the Kronecker delta function,  one has
 \begin{equation}\label{chi}
 \overline{\chi}^G(M_{\underline{\delta}}) = \sum_{[H]\in\Conjsub G}\frac{\vert H\vert}{\vert N_G(H)\vert}
 \left(\sum_{K\in\Sub G}\mu'(H,K) \overline{\chi}(M_{\underline{\delta}})\right)[G/H]\,.
\end{equation}
Let $\omega$ be a holomorphic $G$-invariant 1-form with an isolated singular point at the origin.
If the dimension of the fixed point set $(\CC^n)^K$ of a subgroup $K\subset G$ is greater than $k$, then
the GSV-index $\indGSV(\omega;(\CC^n)^K,0)$ is $(-1)^{n_K-k}$ times the dimension of the space
$\Omega_{V^K,\,\omega}=\Omega^{n_K-k}_{V^K,0}/(\omega\wedge \Omega^{n_K-k-1}_{V^K,0})$,
where $\Omega^{s}_{V^K,0}:=
\Omega^{s}_{(\CC^n)^K,0}/
(f_i\Omega^{s}_{(\CC^n)^K,0}, df_i\wedge \Omega^{s-1}_{(\CC^n)^K,0})$ \cite{MMJ, MZ}.

Propositions~\ref{propsubgroups}, \ref{propGSV-rad} and Equation~\ref{chi} imply that
$$
\indGSV^G(\omega; V,0) = \sum_{[H]\in\Conjsub G}\frac{\vert H\vert}{\vert N_G(H)\vert}
 \left(\sum_{{K\in\Sub G}\atop{\dim{(\CC^K)^K} > k}}
 \mu(H,K) \dim \Omega_{V^K, \,\omega}\right)[G/H]\,.
$$
\end{example}


\bigskip
\noindent Leibniz Universit\"{a}t Hannover, Institut f\"{u}r Algebraische Geometrie,\\
Postfach 6009, D-30060 Hannover, Germany \\
E-mail: ebeling@math.uni-hannover.de\\

\medskip
\noindent Moscow State University, Faculty of Mechanics and Mathematics,\\
Moscow, GSP-1, 119991, Russia\\
E-mail: sabir@mccme.ru

\end{document}